\documentclass[10 pt,a4paper]{article}

\usepackage{amsfonts,amsmath,amsthm,amssymb,graphicx,fancyhdr,makeidx,amscd}
\usepackage[all]{xy}

\newtheorem{Proposition}{Proposition}

\newtheorem{Theorem}{Theorem}
\newtheorem{Remark}{Remark}

\numberwithin{equation}{section} 

\begin{document}

\newcommand{\Hess}{\mathrm{Hess}\, }
\newcommand{\hess}{\mathrm{hess}\, }
\newcommand{\cut}{\mathrm{cut}}
\newcommand{\ind}{\mathrm{ind}}
\newcommand{\ess}{\mathrm{ess}}
\newcommand{\longra}{\longrightarrow}
\newcommand{\eps}{\varepsilon}
\newcommand{\ra}{\rightarrow}
\newcommand{\vol}{\mathrm{vol}}
\newcommand{\Ricc}{\mathrm{Ricc}}
\newcommand{\di}{\mathrm{d}}
\newcommand{\R}{\mathbb R}
\newcommand{\C}{\mathbb C}
\newcommand{\N}{\mathbb N}
\newcommand{\esse}{\mathbb S}
\newcommand{\bull}{\rule{2.5mm}{2.5mm}\vskip 0.5 truecm}
\newcommand{\binomio}[2]{\genfrac{}{}{0pt}{}{#1}{#2}} %identico ad \atop
\newcommand{\metric}{\langle \, , \, \rangle}
\newcommand{\lip}{\mathrm{Lip}}
\newcommand{\loc}{\mathrm{loc}}
\newcommand{\diver}{\mathrm{div}}
\newcommand{\disp}{\displaystyle}
\newcommand{\rad}{\mathrm{rad}}
\newcommand{\mmetric}{\langle\langle \, , \, \rangle\rangle}
\newcommand{\sn}{\mathrm{sn}}
\newcommand{\cn}{\mathrm{cn}}
\newcommand{\ink}{\mathrm{in}}

\newcommand{\rmi}{(\rm i)\,}
\newcommand{\rmii}{(\rm ii)\,}
\newcommand{\rmiii}{(\rm iii)\,}
\newcommand{\rmiv}{(\rm iv)\,}
\newcommand{\rmv}{(\rm v)\,}
\newcommand{\rmvi}{(\rm vi)\,}
\newcommand{\rmvii}{(\rm vii)\,}
\newcommand{\rmviii}{(\rm viii)\,}
%%%%%%%%%
\newcommand{\vp}{\varphi}
\renewcommand{\div}[1]{{\mathop{\mathrm div}}\left(#1\right)}
\newcommand{\divphi}[1]{{\mathop{\mathrm div}}\bigl(\vert \nabla #1
\vert^{-1} \varphi(\vert \nabla #1 \vert)\nabla #1   \bigr)}
\newcommand{\nablaphi}[1]{\vert \nabla #1\vert^{-1}
\varphi(\vert \nabla #1 \vert)\nabla #1}
\newcommand{\modnabla}[1]{\vert \nabla #1\vert }
\newcommand{\modnablaphi}[1]{\varphi\bigl(\vert \nabla #1 \vert\bigr)
\vert \nabla #1\vert }
\renewcommand{\a}{\alpha}
\newcommand{\e}{\epsilon}
\newcommand{\D}{\Delta}
\renewcommand{\d}{\delta}
\newcommand{\ty}{\infty}
\newcommand{\g}{\gamma}
\renewcommand{\b}{\beta}
\newcommand{\ds}{\displaystyle}
\newcommand{\cL}{\mathcal{L}}
\newcommand{\essem}{\mathds{S}^m}
\newcommand{\erre}{\mathds{R}}
\newcommand{\errem}{\mathds{R}^m}
\newcommand{\enne}{\mathds{N}}
\newcommand{\acca}{\mathds{H}}

\newcommand{\cvett}{\Gamma(TM)}
\newcommand{\cinf}{C^{\infty}(M)}
\newcommand{\sptg}[1]{T_{#1}M}
\newcommand{\partder}[1]{\frac{\partial}{\partial {#1}}}
\newcommand{\partderf}[2]{\frac{\partial {#1}}{\partial {#2}}}
\newcommand{\ctloc}{(\mathcal{U}, \varphi)}
\newcommand{\fcoord}{x^1, \ldots, x^n}
\newcommand{\ddk}[2]{\delta_{#2}^{#1}}
\newcommand{\christ}{\Gamma_{ij}^k}
\newcommand{\ricc}{\operatorname{Ricc}}
\newcommand{\supp}{\operatorname{supp}}
\newcommand{\sgn}{\operatorname{sgn}}
\newcommand{\rg}{\operatorname{rg}}
\newcommand{\inv}[1]{{#1}^{-1}}
\newcommand{\id}{\operatorname{id}}
\newcommand{\jacobi}[3]{\sq{\sq{#1,#2},#3}+\sq{\sq{#2,#3},#1}+\sq{\sq{#3,#1},#2}=0}
\newcommand{\lie}{\mathfrak{g}}
\newcommand{\wedgedot}{\wedge\cdots\wedge}
\newcommand{\rp}{\erre\mathds{P}}
\newcommand{\II}{\operatorname{II}}
\newcommand{\gradh}[1]{\nabla_{H^m}{#1}}
\newcommand{\absh}[1]{{\left|#1\right|_{H^m}}}
\newcommand{\mob}{\mathrm{M\ddot{o}b}}
\newcommand{\mab}{\mathfrak{m\ddot{o}b}}
\newcommand{\foc}{\mathrm{foc}}
\newcommand{\F}{\mathcal{F}}
\newcommand{\Cf}{\mathcal{C}_f}
\newcommand{\cutf}{\mathrm{cut}_{f}}
\newcommand{\Cn}{\mathcal{C}_n}
\newcommand{\cutn}{\mathrm{cut}_{n}}
\newcommand{\Ca}{\mathcal{C}_a}
\newcommand{\cuta}{\mathrm{cut}_{a}}
\newcommand{\cutc}{\mathrm{cut}_c}
\newcommand{\cutcf}{\mathrm{cut}_{cf}}
\newcommand{\rk}{\mathrm{rk}}
\newcommand{\crit}{\mathrm{crit}}
\newcommand{\diam}{\mathrm{diam}}
\newcommand{\haus}{\mathcal{H}}
\newcommand{\po}{\mathrm{po}}

\author{Bruno Bianchini \and Luciano Mari \and Marco Rigoli}
\title{\textbf{Some generalizations of Calabi compactness theorem}}
\date{}
\maketitle
\large \begin{center} Fortaleza \\[0.2cm]
\normalsize dedicated to Gervasio Colares for his $80^{\mathrm{th}}$ birthday
\end{center}

\scriptsize
$$
\begin{array}{cc}
\begin{array}{c}
\text{Dipartimento di Matematica Pura e Applicata}\\
\text{Universit\`a degli Studi di Padova}\\
\text{Via Trieste 63}\\
\text{I-35121 Padova, ITALY}\\
\text{e-mail: bianchini@dmsa.unipd.it}
\end{array} & \qquad \begin{array}{c}
\text{Dipartimento di Matematica}\\
\text{Universit\`a degli Studi di Milano}\\
\text{Via Saldini 50}\\
\text{I-20133 Milano, ITALY}\\
\text{e.mail: lucio.mari@libero.it, marco.rigoli@unimi.it}
\end{array}
\end{array}
$$
\vspace{0.5cm}

\begin{abstract}
\footnote{\textbf{Mathematic subject classification 2010}: primary
53C20; secondary 34C10.\par
\textbf{ \ Keywords}: compactness, Myers’ type theorems, oscillation, positioning of zeros.}
In this paper we obtain generalized Calabi-type compactness criteria for complete Riemannian manifolds that allow the presence of negative amounts of Ricci curvature. These, in turn, can be rephrased as new conditions for the positivity, for the existence of a first zero and for the nonoscillatory-oscillatory behaviour of a solution $g(t)$ of $g''+Kg=0$, subjected to the initial condition $g(0)=0$, $g'(0)=1$. A unified approach for this ODE, based on the notion of critical curve, is presented. With the aid of suitable examples, we show that our new criteria are sharp and, even for $K\ge 0$, in borderline cases they improve on previous works of Calabi, Hille-Nehari and Moore. 
\end{abstract}
 \maketitle

\normalsize

\section{Basic comparison and Myers type compactness result}
Hereafter, we consider a connected, complete Riemannian manifold $(M,\metric)$, and a chosen reference origin $o \in M$. Let $D_o = M \backslash (\left\{ o \right\} \cup \mathrm{cut} (o))$ be the maximal domain of normal coordinates centered at $o$, and denote with $r(x)$ the distance function from $o$. The classical Bonnet-Myers theorem, showing the compactness of $M$ under the condition
\begin{equation} \label{uno}
\mathrm{Ricc} \ge (m-1) B^2 \metric  
\end{equation}
for some $B>0$, can be proved as a consequence of the Laplacian comparison theorem. Indeed, let us recall the following generalized form of this latter.
\begin{Theorem}[\textbf{Theorem 2.4 of \cite{prs}}]\label{teo_laplacianodasopra}
Let $M$ be as above. Assume that the radial Ricci curvature satisfies 
\begin{equation}\label{richiericci}
\mathrm{Ricc}(\nabla r,\nabla r)(x) \ge -(m-1)G(r(x)) \qquad
\text{on } \ M,
\end{equation}
for some function $G\in C^0(\R^+_0)$, and let $g\in C^2(\R^+_0)$
be a solution of
\begin{equation}\label{oohh2}
\left\{ \begin{array}{l} g''-Gg \ge 0 \\[0.2cm]
g(0)=0, \quad g'(0)=1. \end{array}\right.
\end{equation}
Let $(0,R_0)$ (possibly $R_0 =+\infty$) be the maximal interval where $g$ is positive. Then,
\begin{equation}\label{dentroBR}
D_o\subset B_{R_0}
\end{equation}
and the inequality
\begin{equation}\label{lasop2}
\Delta r(x) \le (m-1)\frac{g'(r(x))}{g(r(x))}
\end{equation}
holds pointwise on $D_o$ and weakly on $M$.
\end{Theorem}

Suppose the validity of (\ref{uno}) so that $G(t)=-B^2$. A simple checking shows that $g(t)=B^{-1} \sin(Bt)$ solves (\ref{oohh2}). Its first positive zero is at $2\pi/B$. Then \eqref{dentroBR} gives  that $\overline{D_o} \equiv M$ is bounded. Since $M$ is closed, the Hopf-Rinow theorem implies that $M$ is compact. In fact, we have also shown that $\diam(M) \le 2\pi/B$, but since \eqref{uno} is indipendent of the origin $o$ we can improve the above to the sharp estimate $\diam(M) \le \pi/B$. \par 
Cleary the key point of our proof lies in the validity of the inclusion $D_o \subset B_{R_o}$. The way to prove this latter is as follows. Suppose to have shown (\ref{lasop2}) on $D_o \cap B_{R_o}$

A computation in normal coordinates gives
$$
\Delta r = \frac{\partial}{\partial r} \log \sqrt{\tilde g(r,\theta)},
$$
where $\tilde g(r,\theta)$ is the determinant of the metric in this
coordinate system. Thus, \eqref{lasop2} on $D_o\cap B_{R_0}$
reads
\begin{equation}\label{cinque}
\frac{\partial}{\partial r} \log \sqrt{\tilde g(r,\theta)} \le
(m-1)\frac{g'(r)}{g(r)}.
\end{equation}
Fix the unit vector $\theta$ and let $\gamma_\theta$ be the unit speed geodesic emanating from $o$ with $\dot{\gamma}_\theta(o)=\theta$. $\gamma_\theta$ will stop to be minimizing after the first cut point attained at $t=c(\theta)>0$. With $\epsilon >0$ sufficiently small, we integrate (\ref{cinque}) on $[\epsilon,{\rm min}\{c(\theta),R_o \}]$, we let $\epsilon \to 0^+$  and we use the asymptotic behaviours in $0$ to get
$$
\sqrt{\tilde g(r,\theta)}\le g(r)^{m-1},
$$
Since $\tilde g(r,\theta)>0$ on
$D_o$, we have $R_0\ge c(\theta)$, that is, $D_o\subset
B_{R_0}$. \par 
However, by a result of M. Morse, a complete manifold $M$ is compact if and only if each unit speed geodesic $\gamma_\theta$ emanating from some fixed origin $o$ ceases to be a segment i.e. length minimizing, for a value $c(t_o)$ of its parameter $t$ which is finite. Thus, the above reasoning appears to be slightly redundant, in the sense that it provides a bound $R_0$
which is independent of the considered unit speed geodesic from $o$. This motivates the following result of Galloway \cite{galloway}.

\begin{Theorem}\label{teo_galloway}
Let $(M,\metric)$ be a complete Riemannian manifold of dimension
$m\ge 2$. Assume that, for some origin $o$ and for every unit
speed geodesic $\gamma: \R^+_0\ra M$ emanating from $o$, the
solution $g$ of
\begin{equation}\label{eq_galloway}
\left\{\begin{array}{l} g'' + \dfrac{\mathrm{Ricc}(\gamma',
\gamma')(t)}{m-1}g = 0,
\\[0,4cm]
g(0)=0, \quad g'(0)=1
\end{array}\right.
\end{equation}
has a first positive zero. Then, $M$ is compact with finite fundamental
group.
\end{Theorem}
\begin{proof}
Let $r_0 >0$ be the first positive zero of $g$ solution of (\ref{eq_galloway}). Multiply the equation in (\ref{eq_galloway}) by $g$, integrate by parts and use the initial conditions to get 
\begin{equation} \label{sette}
\int_0^{r_0} (g')^2-\int_0^{r_0} \frac{\rm{Ricc}(\dot{\gamma},\dot{\gamma})}{m-1} g^2 =0
\end{equation}
By Rayleigh characterization, this means that the operator 
$$
L=\frac{\di^2}{\di t^2}+\frac{\rm{Ricc}(\dot{\gamma},\dot{\gamma})}{m-1}
$$
satisfies 
$$
\lambda_1^L ([0,r_0]) \le 0,
$$
and by monotonicity of eigenvalue
$$
\lambda_1^L ([0,r])<0 \qquad \forall \ r > r_0.
$$
But $L$ is the stability operator for the geodesic $\gamma$, and on $[0,T]$ $\gamma$ is minimizing only if
$$
\lambda_1^L([0,T]) \ge 0.
$$
Thus if the value $c(\gamma)$ gives the cut-point di $o$ along $\gamma$ it must be $c(\gamma) \le r_0$. By Morse result $M$ is compact. The same procedure can also be applied to the Riemannian universal covering $\widetilde{M} \ra M$, showing that $\widetilde{M}$ is compact and thus that $\Pi_1(M)$ is finite. 
\end{proof}
If we ignore that $L$ is the stability operator for the unit speed geodesic $\gamma$ we can proceed with the following analytic alternative proof. \par 
Let $p \in D_o$, and let $\gamma : [0,r(p)] \ra M $ be the minimizing geodesic from $o$ to $p$ so that $r(\gamma(t))=t$ and $\nabla r \circ \gamma =\dot\gamma$ for $t \in [0,r(p)]$. We fix a local orthonormal coframe $\{\theta^i\}$ to perform computations.  Here  $1\le i,j, \ldots\le m$ and we use Einstein summation convention. Then for the distance function $r$ on $D_o$ we have
$$
\di r=r_i\theta^i,
$$
and Gauss lemma writes
\begin{equation}\label{gauss}
r_ir_i\equiv 1.
\end{equation}
Taking covariant derivative of (\ref{gauss}) we obtain
\begin{equation}\label{hessianosugradr}
r_{ij}r_i = 0  
\end{equation}
that is,
\begin{equation}\label{dieci}
\Hess r(\nabla r,
\cdot)=0.
\end{equation}
Covariant differentiation of \eqref{hessianosugradr} yields
\begin{equation} \label{undici}
r_{ijk}r_i + r_{ij} r_{ik} =0.
\end{equation}
From the simmetry $r_{ij}=r_{ji}$ we deduce that $r_{ijk}=r_{jik}$, and by the Ricci commutation rules
$$
r_{ijk}=r_{ikj} + r_t R_{tijk}  
$$
$R_{tijk}$ the components of the Riemann tensor. Using this in (\ref{undici}) we get
$$
0= r_{ijk}r_i + r_{ij} r_{ik} = r_{jik}r_i + r_{ij} r_{ik} =
r_{jki}r_i + r_tR^t_{jik}r_i + r_{ij}r_{ik}.
$$
Thus, tracing with respect to $j$ and $k$
$$
r_ir_{kki}+r_tr_i R_{ti} +r_{ik}r_{ik}=0,
$$
with $R_{ti}$ the components of the Ricci tensor. In other words
$$
\left<\nabla\Delta r, \nabla r \right>+{\rm Ricc} \left(\nabla r, \nabla r \right)+ \left| {\rm  Hess} (r) \right|^2 =0
$$
Computing along $\gamma$
$$
\frac{\di}{\di t} \left(\Delta r \circ \gamma \right)+ \left| {\rm  Hess} (r) \right|^2+ {\rm Ricc} \left(\nabla r, \nabla r \right)=0
$$
on $[0,r(p)]$. Using (\ref{dieci}) and Newton's inequality, we have
$$
\left| {\rm Hess}(r)\right|^2 \ge \frac{(\Delta r)^2}{m-1},
$$
and setting $\varphi(t)=\Delta r \circ \gamma(t)$ from the above we obtain
\begin{equation}\label{dodici}
\frac{\di}{\di t} \varphi(t)+ \frac{\varphi(t)^2}{m-1}+ \Ricc \left(\nabla r, \nabla r \right) \le 0
\end{equation}
on $[0,r(p)]$. Furthermore, it is well known that
$$
\Delta r = \frac{m-1}{r} + o(1) \qquad {\rm as}\ r \to 0^+
$$
Hence, since $\gamma$ is minimizing
\begin{equation}\label{tredici}
\frac{1}{m-1} \varphi (t) =\frac{1}{(r \circ \gamma)(t)}+o(1)=\frac{1}{t} +o(1) \qquad {\rm as} \ t \to 0^+
\end{equation}
Defining 
$$
u(t)=t \exp\left\{\int_0^t \left(\frac{\varphi(s)}{m-1}-\frac{1}{s} \right) \di s\right\}
$$
on $[0,r(p)]$, $u$ is well defined because of \eqref{tredici} and a computation using (\ref{dodici}) gives
\begin{equation}\label{quattordici}
\frac{\di^2}{\di t^2} u+  \frac{{\rm Ricc}(\dot{\gamma},\dot{\gamma})}{m-1} u \le 0
\end{equation}

Let now $h$ be any $C^1 ([0,r(p)])$ function such that $h(0)=0=h(r(p))$. Since $u>0$ on $(0,r(p)]$ the function $h^2 u'/u$ is well defined on $(0,r(p)]$. Differentiating, using (\ref{quattordici}) and Young inequality we get
$$
\frac{\di}{\di t}\left( h^2 \frac{u'}{u} \right) \le -\frac{{\rm Ricc(\dot{\gamma},\dot{\gamma})}}{m-1} h^2-h^2\left( \frac{u'}{u} \right)^2+2hh'\frac{u'}{u} \le -\frac{{\rm Ricc(\dot{\gamma},\dot{\gamma})}}{m-1} h^2+(h')^2
$$
Fix $\epsilon >0$ sufficiently small. Integration of the above on $[\epsilon,r(p)]$ gives
$$
-h^2(\epsilon) \frac{u'(\epsilon)}{u(\epsilon)} \le \int_\epsilon^{r(p)} (h')^2- \int_\epsilon^{r(p)} \frac{{\rm Ricc(\dot{\gamma},\dot{\gamma})}}{m-1} h^2
$$
Since $h(\epsilon)=A\epsilon+o(1)$, for $\epsilon \to 0^+$ where $A \in \R$, letting $\epsilon \to 0^+$ we obtain
\begin{equation}\label{quindici}
\int_0^{r(p)} (h')^2- \int_0^{r(p)} \frac{{\rm Ricc(\dot{\gamma},\dot{\gamma})}}{m-1} h^2 \ge 0
\end{equation}
This contradicts (\ref{sette}) unless $r(p) \le r_0$. \par 
Thus we have reduced the compactness problem for the complete manifold $M$ to the problem of the existence of a first zero for solutions of the Cauchy problem
\begin{equation}\tag{CP}\label{sedici}
\left\{ \begin{array}{l}
g'' +K(t) g=0 \qquad \text{on } \R^+\\[0.2cm]
g(0)=0, \quad g'(0)=1.
\end{array}\right.
\end{equation}
where in our geometric application
\begin{equation}\label{diciasette}
K(t)=K_\gamma (t)=\frac{{\rm Ricc(\dot{\gamma},\dot{\gamma})}}{m-1} (t)
\end{equation}
We observe that the existence of a first zero is also "a posteriori" guaranteed via an oscillation result for the same equation, and that uniform upper estimate for the positioning of the first zero yields a diameter estimate. In this perspective the original result of Calabi can be stated as follows (see also Theorem 3.11 of \cite{bmr2}).

\begin{Theorem}[\textbf{Theorems 1 and
2 of \cite{calabi}}]\label{ilverocalabi}
Let $M$ be as above, and assume that $\Ricc \ge 0$ on $M$. Suppose that for each unit
speed geodesic $\gamma$ emanating from $o$ there exist $0<a<b$, possibly depending on $\gamma$, such that
\begin{equation}\label{calabialfinito}
\int_a^b \sqrt{\frac{\mathrm{Ricc}(\gamma',\gamma')(s)}{m-1}}\di s > \left\{ \left(1 +\frac12 \log\frac{b}{a}\right)^2 -1\right\}^{1/2}.
\end{equation}
Then, $M$ is compact and has finite fundamental group. In particular, this holds provided that
\begin{equation}\label{laprimacalabi}
\limsup_{t\ra +\infty} \left(\int_1^t \sqrt{\frac{\mathrm{Ricc}(\gamma',\gamma')(s)}{m-1}}\di
s - \frac{1}{2} \log t\right) = +\infty.
\end{equation}
\end{Theorem}
\begin{Remark}
\emph{As a matter of fact, under the assumption $\Ricc\ge 0$ on $M$, \eqref{laprimacalabi} gives an oscillation result for \eqref{sedici}.}
\end{Remark}
In Calabi result the requirement $\Ricc \ge 0$ is essential. We stress that \eqref{calabialfinito} is, to the best of our knowledge, the first instance of a condition in finite form for the existence of a first zero, that
is, a condition involving the potential $K$ only in a compact
interval $[a,b]$.
One of the main purpose of the present paper is to extend the result even when Ricci is negative somewhere. It shall be observed that the problem of obtaining Myers type compactness theorems under the presence of a suitably small amount of negative Ricci curvature has already been a flourishing field of research, for which we refer the reader to  \cite{jywu}, \cite{elworthyrosenberg}, \cite{rosenbergyang} and the references therein. However, the techniques employed in these papers are of various nature and neither of them relies on oscillation type results for a linear ODE, nor it gives explicit bounds for the amount of negative curvature allowed. Indeed, it should be pointed out that the method in \cite{jywu} via Jacobi fields is not distant from our approach. A much closely related result is the recent  \cite{gionamichipaolo}, where the case $\mathrm{Ricc} \ge -B^2$ is analyzed.
 
\section{The role of the critical curve}
As we will see shortly, in order to extend Calabi result, we shall deal with a slightly different ODE. In particular, we are concerned with the following problems:
\begin{itemize}
\item[i)] study the existence of a first zero of solutions $z(r)$ of
\begin{equation}\label{cauchypr}
\left\{ \begin{array}{l} (v(r)z'(r))' + A(r)v(r) z(r) = 0 \qquad
\text{on } \  \R^+ \\[0.2cm]
z(0^+) = z_0>0,  
\end{array}\right.
\end{equation}
with $A(t) \ge 0$, $v(t) >0$ on $\R^+$;
\item[ii)] give an upper bound for the positioning of the first zero of $z$;
\item[iii)] study the oscillatory behavior of \eqref{cauchypr};
\item[iv)] extend the obtained result when $A(r)$ changes sign.
\end{itemize}
Towards these aims we introduce the "critical curve" $\chi (r)$ relative to \eqref{cauchypr}
or to the next Cauchy problem
\begin{equation}\label{cauchyprinf}
\left\{ \begin{array}{l} (v(r)z'(r))' + A(r)v(r) z(r) = 0 \qquad
\text{on } \  [r_0,+\infty), \quad r_0>0\\[0.2cm]
z(r_0^+) = z_0 \in \R,  
\end{array}\right.
\end{equation}
To do this we require the assumptions
 \begin{equation}\label{sesto}\tag{V1}
0 \le v(r) \in L^\infty_{\mathrm{loc}}(\R_0^+), \qquad
\frac{1}{v(r)} \in L^\infty_{\mathrm{loc}}(\R^+), \qquad
\lim_{r\ra 0^+} v(r) =0
 \end{equation}
(the last equation request is intended on a rapresentative of $v$) and the integrability condition
\begin{equation}\label{VL1} \tag{V$_{\text{L1}}$}
\frac{1}{v(r)} \in L^1(+\infty).
\end{equation}
We set
\begin{equation}\label{defchi}
\chi(r) = \left\{2v(r) \int_r^{+\infty} \frac{\di
s}{v(s)}\right\}^{-2} = \left\{\left(-\dfrac 12
\log\int_r^{+\infty}\dfrac{\di s}{v(s)}\right)'\right\}^2 
\end{equation}
%Note that if $v(r)$ is increasing near $0^+$ then one can show that
%$$
%\chi (0^+) =+\infty
%$$
Fix $0<R<r$, from the definition di $\chi$ we deduce

 \begin{equation}\label{notaint}
\int^r_R \sqrt{\chi(s)}\di s = \frac 12 \log
\left\{\left(\int_R^{+\infty} \frac{\di s}{v(s)}\right)\Big/\left(
\int_r^{+\infty} \frac{\di s}{v(s)}\right)\right\} \qquad \forall
\ 0<R<r,
\end{equation}
Thus letting $r\ra +\infty$, we obtain
\begin{equation}\label{chinoninL1}
\sqrt{\chi(r)} \not\in L^1(+\infty)
\end{equation}
It is worth to stress that the function $\chi$ only depends on the weight $v$, not on $A$. Note that, although \eqref{sedici} can be thought as a version of \eqref{cauchypr} with $v\equiv 1$, assumptions \eqref{sesto}, \eqref{VL1} are not satisfied. Thus, the next main Theorem \ref{main} below cannot be directly applied to \eqref{sedici}. 

The study of the Cauchy problem \eqref{cauchypr} turns out to be extremely useful in a number of different geometric problems, not only those described in this paper. For instance, a mainstream application of it is to derive spectral estimates for stationary Schr\"odinger tipe operators via radialization techniques. In this case, the role of $v$ is played by the volume growth of geodesic spheres centered at $o$, for which \eqref{sesto} is the highest regularity that we can in general guarantee. However, since there are natural upper and lower bounds coming from the Laplacian comparison theorems, it is worth to relate the critical curve with that of, say, an upper bound for $v$. More precisely, for $f$ satisfying 
\begin{equation}\tag{F1} \label{F1}
\disp f \in L^\infty_{\rm loc} (\R_0^+), \qquad \frac{1}{f} \in L^\infty_{\rm loc} (\R^+), \qquad 0 \le v \le f \quad {\rm on} \ \R_0^+ 
\end{equation} 
\begin{equation}\tag{F$_{\text{L1}}$}\label{FL1}
\frac 1 f \in L^1(+\infty)
\end{equation}
we shall compare $\chi(r)$ with the critical curve $\chi_f (r)$ defined again via \eqref{defchi}. We observe that, for any positive constant $c$, $\chi_{cf}=\chi_{f}$. This suggests that, in general, $v \le f$ does not imply $\chi \le \chi_f$. To recover this property we need a more stringent relation between $v$ and $f$.

\begin{Proposition}[\textbf{Proposition 4.13 of \cite{bmr2}}]\label{confront}
Let $v,f$ satisfy \eqref{sesto}, \eqref{VL1} on some  interval $I=
(r_0,+\infty)\subset \R^+$. Then,
\begin{itemize}
\item[(i)] If $v/f$ is non-increasing on $I$, $\chi(r)\le
\chi_f(r)$ on $I$;
\item[(ii)] If $v/f$ is non-decreasing on $I$, $\chi(r)\ge
\chi_f(r)$ on $I$;
\end{itemize}
\end{Proposition}
In the case $v(r) =\vol(\partial B_r)$, the above proposition fits well with the Bishop-Gromov comparison theorem for volumes (\cite{prs}, Theorem 2.14). The interested reader may consult Chapter 4 of \cite{bmr2}, where the authors give a detailed discussion on the critical curve, together with estimates on $\chi$ when $v(r)=\vol(\partial B_r)$, explicit examples, and many applications. For instance, the deep relationship between $\chi(r)$ and optimal weights for Hardy inequalities is discussed. Since, as we will see, in dealing with Calabi-type compactness results the role of $v$ will be played by some suitable weight which has no direct relation with volumes, we shall not pursue this line of argument any further.

We now list the assumptions under which  we will treat either of the Cauchy problems \eqref{cauchypr} or \eqref{cauchyprinf}. 
\begin{equation}\tag{V2} \label{V2}
v(r) \int_r^a \frac{\di s}{v(s)}; \qquad \frac{1}{v(r)} \int_o^r v(s) \di s
\ \in L^\infty ([0,a])
\end{equation}
for some $a \in \R^+$.
\begin{equation} \tag{V3} \label{V3}
\frac{1}{v(r)} \int_0^r v(s) \di s = o(1) \qquad {\rm as} \quad r \to 0^+
\end{equation}  
\begin{equation}\tag{A1}\label{A1}
A(r) \in L^\infty_{\rm loc} (\R_0^+)
\end{equation}
Conditions
\begin{itemize}
\item[1.] (\ref{A1}), (\ref{sesto}), (\ref{V2}) and (\ref{V3}) guarantee the existence of a solution $z \in \lip_\loc(\R^+_0)$ of (\ref{cauchypr})
\item[2.] (\ref{A1}), (\ref{sesto}) and (\ref{V2}) its uniqueness
\item[3.] (\ref{A1}), (\ref{sesto}) the fact that each solution $z \not\equiv 0$ has isolated zeros, if any.
\end{itemize}
Note that \eqref{V2} and \eqref{V3} are automatically satisfied if $v(r)$ is non-decreasing in a neighbourhood of $0$.\\ 
The following theorem summarizes some of the results obtained in \cite{bmr}. 
\begin{Theorem}\label{main}
Let \eqref{A1}, \eqref{sesto}, \eqref{F1}, \eqref{VL1} be met, and let $z \in \lip_\loc(\R_0^+)$ be a solution of
\begin{equation}\label{cauchy2}
\left\{ \begin{array}{l} (v(r)z'(r))' + A(r)v(r) z(r) = 0 \qquad
\text{on } \ \R^+,\\[0.2cm]
z(0^+) = z_0>0.
\end{array}\right.
\end{equation}
Then,
\begin{itemize}
\item[(1)] [\emph{\textbf{Theorem 5.2 of \cite{bmr}}}] If $A(r)\le \chi(r)$ on $\R^+$, then $z>0$ on $\R^+$. Furthermore, there exists $r_1 >0$ and a constant $C=C(r_1)>0$ such that 
\begin{equation}\label{lowerbound}
z(r) \ge \disp -C \sqrt{\int_r^{+\infty} \dfrac{\di
s}{f(s)}}\log \int_r^{+\infty} \dfrac{\di s}{f(s)} \qquad \text{on } [r_1,+\infty).
\end{equation}
\item[(2)] [\emph{\textbf{Corollary 5.4 of \cite{bmr}}}] If $A(r) \le \chi(r)$ on $[r_0,+\infty)$, for some $r_0>0$, then $z$ is nonoscillatory, that is, it has only finitely many zeroes (if any).
\item[(3)] [\emph{\textbf{Corollary 6.3 of \cite{bmr}}}] If $A \ge 0$ on $\R^+$, $A\not\equiv 0$ and there exist $r>R>0$ such that $A\not \equiv 0$ on $[0,R]$ and 
\begin{equation}\label{contrad}
\begin{array}{l}
\disp \int_{R}^r \left(\sqrt{A(s)}-\sqrt{\chi_{f}(s)}\right)\di
s > -\dfrac{1}{2} \left( \log \int_0^{R} A(s)v(s)\di s + \log
\int_{R}^{+\infty} \frac{\di s}{f(s)} \right)
\end{array}
\end{equation}
then $z$ has
a first zero. Moreover, this is attained on $(0,\overline{R}]$,
where $\overline{R}>0$ is the unique real number satisfying
\begin{equation}\label{prim}
\int^{r}_{R} {\sqrt{A(s)} \di s} = -\frac{1}{2} \log
\int_0^{R} A(s)v(s)\di s -\frac{1}{2} \log
\int_{r}^{\overline{R}} {\frac{\di s}{f(s)}}
\end{equation}
\item[(4)] [\emph{\textbf{Theorem 6.6 of \cite{bmr}}}] If $A\ge 0$ on $\R^+$ and, for some (hence any) $R>0$ such that $A\not\equiv 0$ on $[0,R]$, 
\begin{equation}\label{condforte}
\limsup_{r\rightarrow +\infty} \int_R^r
\left(\sqrt{A(s)}-\sqrt{\chi_f(s)}\right) \di s = +\infty
\end{equation}
then $z$ is oscillatory, that is, it has infinitely many zeroes.
\end{itemize}
\end{Theorem}

\begin{Remark} 
\emph{In fact, for $(2)$ and $(4)$ to hold, it is enough that $z$ solves the Cauchy problem only on $[r_0,+\infty)$, for some $r_0>0$ and for some initial condition $z(r_0)$, $(vz')(r_0)$.} 
\end{Remark}

It is worth to make some observations on the conditions in the above theorem. 
\begin{itemize}
\item[-] In $(1)$, $A\le \chi$ cannot be replaced with $A\le \chi_f$. The reason is that, as already observed, no relations between $\chi$ and $\chi_f$ can be deduced from the sole requirement $v\le f$ in \eqref{F1}. However, note that $\chi_f$ appears both in \eqref{contrad} and in \eqref{condforte}. This is due to the technique developed for $(3)$ and $(4)$, which is different from that used for $(1)$ and $(2)$. 
\item[-] The lower bound \eqref{lowerbound} is sharp. Indeed, it can be showed that if $z$ is positive on $\R^+$ and $A\ge \chi$ on some $[r_0,+\infty)$, then necessarily $z$ is bounded from above by the quantity on the RHS of \eqref{lowerbound}, for some $C>0$.
\item[-] The right hand side of
\eqref{contrad} is independent both of $r$ and of the behavior of
$A$ after $R$. Therefore, the left hand side of \eqref{contrad} represents how much must $A$
exceed a critical curve modelled on $f$ in the compact region
$[R,r]$ in order to have a first zero for $z$, and it only
depends on the behavior of $A$ and $f$ before $R$ (the first
addendum of the RHS), and on the growth of $f$ after $R$. This is conceptually simpler than Calabi compactness condition, where the role of $a,b$ is balanced between the two sides of \eqref{calabialfinito}.
\end{itemize}
\begin{Remark}\label{indeboliamo}
\emph{The assumptions in  $(3)$ and $(4)$ can be weakened. Indeed, it is enough that $z$ solves the inequality $(vz')'+Avz \le 0$ on $\R^+$, and that its initial condition satisfies 
$$
\frac{vz'}{z}(0^+)=0.
$$
Note that sufficiently mild singularities of $z$ as $r\ra 0^+$ are allowed, depending on the order of zero of $v(r)$ at $0$.}
\end{Remark}
\begin{Remark}\label{remimpo}
\emph{Using \eqref{notaint} we see that \eqref{condforte} can be equivalently expressed as
\begin{equation}\label{nuova}
\limsup_{r \to +\infty} \left\{ \int_R^r \sqrt{A(s)} + \frac{1}{2} \log \int_r^{+\infty} \frac{\di s}{f(s)} \right\}=+\infty.
\end{equation}
}
\end{Remark}\par
The similarity between \eqref{nuova} and \eqref{laprimacalabi} is evident. Indeed, as a first application of Theorem \ref{main} let us show that Calabi condition \eqref{laprimacalabi} implies that the solution of  \eqref{sedici}, with 
 \begin{equation}\label{problem}
%  \left\{
 % \begin{array}{l}
 % g''+K(t) g=0 \\[0.2cm]
 % g(0)=0, \quad g'(0)=1
 % \end{array}
 % \right. ,  \qquad 
K(t)=\frac{{\Ricc}(\dot{\gamma},\dot{\gamma})}{m-1}(t) \ge 0,
  \end{equation}
is oscillatory. \par 
Indeed, choose any $v$ satisfying \eqref{sesto}, \eqref{V2}, \eqref{V3} and $v^{-1} \in L^1(+\infty) \backslash L^1 (0^+)$, for instance $v(r)=r^{m-1}$ for some $m\ge 3$. Let $r=r(t)$ be the inverse function of
 \begin{equation}\label{tr}
 t(r)=\left(\int_r^{+\infty }\frac{\di s}{v(s)} \right)^{-1}
\end{equation}
and define
\begin{equation}\label{zr}
z(r)=\frac{g(t(r))}{t(r)}
\end{equation}
Then $z$ solves
\begin{equation}\label{problemA}
\left\{
\begin{array}{l}
\disp (vz')'+\frac{K(t(r))t^4(r)}{v^2(r)} v(r) z=0 \qquad {\rm on} \ \R^+\\[0.4cm]
z(0)=1 \qquad (vz')(0)=0
\end{array}
\right.
\end{equation}
where now differentiation is with respect to the variable $r$. If \eqref{nuova} holds with $f=v$ and
$$
A(r)=\frac{K(t(r))t^4(r)}{v^2(r)} \ge 0,
$$
then $z$ oscillates and so does $g$. A change of variables shows that \eqref{nuova} is exactly \eqref{laprimacalabi}. \\
\par
The literature on the qualitative properties of solutions of \eqref{sedici} is enormous, and considerable steps towards the comprehension of the matter have been made throughout all of the $20^{\mathrm{th}}$ century. In particular, a number of sharp oscillatory and nonoscillatory conditions for $g$ have been found. Here, we only quote two of the finest. The first is the so-called Hille-Nehari criterion, see \cite{swanson}, p.45 and \cite{hille}, Theorem $5$ and Corollary $1$. 
\begin{Theorem}\label{teo_HilleNehari}
Let $K \in C^0(\R)\cap L^1(+\infty)$ be non-negative, and consider a solution $g$ of $g'' +Kg = 0$. Denote with $k(t)$, $k_*$ and $k^*$
respectively the quantities
$$
k(t) = t\int_t^{+\infty}K(s)\di s, \qquad k_* =
\liminf_{t\ra +\infty}k(t), \qquad k^*= \limsup_{t\ra +\infty}
k(t).
$$
We have: 
\begin{itemize}
\item[-] if $g$ is nonoscillatory, then necessarily $k_*\le
1/4$ and $k^*\le 1$;
\item[-] if $k(t)\le 1/4$ for $t$ large enough, in particular
if $k^*< 1/4$, then $g$ is nonoscillatory.
\end{itemize}
As a consequence, $k_*> 1/4$ is a sufficient condition for $g$ to be oscillatory.
\end{Theorem}
\begin{Remark}\label{rem_fite}
\emph{If $K\not\in L^1(+\infty)$, the result applies with
$k_*=k^*=+ \infty$, and $g$ is thus oscillatory.
This case is due to W.B. Fite \cite{fite}.}
\end{Remark}
\begin{Remark}
\emph{Improving on an old criterion of Kneser, it can be showed (see \cite{bmr2}, Proposition 2.23) that if $k(t) \le 1/4$ on the whole $\R^+$, then the solution $g$ of \eqref{sedici} is positive and increasing on $\R^+$.}
\end{Remark}

\begin{Remark}
\emph{Hille-Nehari criterion detects the oscillation of $g$ when $K(t) \ge B^2/(1+t^2)$ on $\R^+$, for some $B>1/2$. In a geometrical context, this particular case has been investigated in \cite{cheegergromovtaylor}, where the authors have also obtained upper bounds for the first zero of $g$ solving \eqref{sedici}.}  
\end{Remark}

\begin{Remark}\label{rem_euler0}
\emph{For every $B\in [0,1/2]$, the Cauchy problem associated to the Euler
equation
$$
\left\{ \begin{array}{l} g'' + \dfrac{B^2}{(1+t)^2} g = 0, \\[0.4cm]
g(0)=0, \quad g'(0)=1, \end{array}\right.
$$
has the explicit, positive solution
$$
g(s) =\left\{
\begin{array}{ll} \sqrt{1+t}\log(1+t) & \quad \text{if } \ B=1/2;
\\[0.3cm]
\disp \frac{1}{\sqrt{1-4B^2}}\Big((1+t)^{B''} - (1+t)^{1-B''}\Big)
&
 \quad \text{if } \ B \in [0,1/2),
\end{array}\right.
$$
where
$$
B'' = \frac{1+\sqrt{1-4B^2}}{2} \in (1/2,1]
$$
(see \cite{swanson}, p.45). For $B=1/2$, this example shows
that Hille-Nehari criterion is sharp.}
\end{Remark}

When $k_*= k^* = 1/4$, Hille-Nehari criterion cannot grasp the behaviour of $g$. As we shall see, combining $(2)$ and $(4)$ of Theorem \ref{main} in an iterative way, we can construct sharper and sharper oscillation and nonoscillation criteria that can detect the behaviour of $g$ even in some cases when the Hille-Nehari theorem fails to give information.\par
The second result we quote allows sign-changing potentials $K$ and is due to R. Moore (see \cite{moore}, Theorem 2)

\begin{Theorem}\label{moorenehari}
Let $K\in C^0(\R)$. Each solution $g$ of $g'' + Kg=0$ is oscillatory provided
that, for some $\lambda\in [0,1)$, there exists
\begin{equation}\label{condimoore}
\lim_{t\ra +\infty} \int_0^t s^\lambda K(s)\di s =
+\infty,
\end{equation}
\end{Theorem}
\begin{Remark}
\emph{Setting $\lambda=0$ in Moore statement we recover a result of W. Ambrose \cite{ambrose} and A. Wintner \cite{wintner} (one can also consult \cite{guimaraes}, Corollaries
3.5 and 3.6 for a different proof and a generalization). Remark \ref{rem_euler0} shows that in Moore result the interval of the parameter $\lambda$ cannot be extended to $[0,1]$. Thus, Euler equation suggests that, when restricted to the case $K\ge 0$, Moore criterion is somehow weaker than that of Hille-Nehari.}
\end{Remark}
Another observation on Moore  result is that, although sharp from many points of view, it requires
that the negative part of $K$ be, loosely speaking, globally
smaller than the positive part. This is the essence of the
existence of the limit in \eqref{condimoore}. One of our goal in the next section will be to obtain an oscillation criterion that allows $K$ to have a relevant negative part. Furthermore, with the aid of \eqref{contrad}, we will also find a condition in finite form for the existence of a first zero that allows $K$ to be negative somewhere. As far as we know, there is still no result in this direction besides some very recent work of P. Mastrolia, G. Veronelli and M. Rimoldi, which we recall here for the sake of completeness.
\begin{Theorem}[\textbf{Theorem 5 of \cite{gionamichipaolo}}]\label{pmg}
Suppose that $K \in L^\infty(\R^+_0)$ satisfies $K\ge -B^2$, for some $B\ge 0$, and let $g$ be a solution of \eqref{sedici}. Suppose that there exist $0<a<b$ and
$\lambda\neq 1$ for which either
\begin{equation}\label{caselambda1}
\int_a^b s K_\gamma(s)\di s >
B\left\{b+a\frac{e^{2Ba}+1}{e^{2Ba}-1}\right\} + \frac 14
\log\left(\frac ba\right)
\end{equation}
or
\begin{equation}\label{caselambdaneq1}
\int_a^b s^\lambda K_\gamma(s)\di s >
B\left\{b^\lambda+a^\lambda\frac{e^{2Ba}+1}{e^{2Ba}-1}\right\} +
\frac{\lambda^2}{4(1-\lambda)}\left\{a^{\lambda-1}-b^{\lambda-1}\right\}
\end{equation}
holds (if $B=0$, this has to be intended in a limit sense). Then, $g$ has a first zero.
\end{Theorem}

\begin{Remark}
\emph{The case $B=0$ of the above result is due to Z. Nehari, see \cite{nehari}, p.432 (8), with an entirely different proof. We point out that, in \cite{gionamichipaolo}, the authors also give an upper bound for the position of the first zero.} 
\end{Remark}

%\begin{Theorem}[\textbf{Oscillatory behaviour, see Theorem 6.6 of \cite{bmr2}}]\label{corforte}
%Assume that \eqref{A1}, \eqref{sesto}, \eqref{F1}, $A\ge 0$ hold
%on $[r_0,+\infty)$, for some $r_0\ge 0$. Let $z_0 \in
%\mathbb{R}\backslash \{0\}$. Suppose that either
%\begin{equation}\label{casononint}
%\frac{1}{f(r)} \not\in L^1(+\infty) \qquad \text{and} \qquad
%A(r)v(r) \not \in L^1(+\infty)
%\end{equation}
%or
%\begin{equation}\label{condforte}
%\frac{1}{v(r)} \in L^1(+\infty) \qquad \text{and} \qquad
%\limsup_{r\rightarrow +\infty} \int_R^r
%\left(\sqrt{A(s)}-\sqrt{\chi_f(s)}\right) \di s = +\infty
%\end{equation}
%for some (hence any) $R>r_0$. Then, every solution $z(r)\in
%\lip_\loc([r_0,+\infty))$ of
%\begin{equation}\label{cauchy2inf}
%\left\{ \begin{array}{l} (v(r)z'(r))' + A(r)v(r) z(r) = 0 \qquad
%\text{on } \ (r_0,+\infty)\\[0.2cm]
%z(r_0) = z_0
%\end{array}\right.
%\end{equation}
%is oscillatory.
%\end{Theorem}

%\begin{Remark} There are many well known oscillation results in the litterature. One of the finest is due to Hille and Nehari. \par 
%Given $K \in C^0(\R)\cap L^1(+\infty)$, $K \ge 0$, consider the ODE
%$$
%g'' +Kg = 0
%$$
%If 
%$$
%K_*=\liminf_{t \to +\infty} t \int_t^{+\infty} K(s) > \frac{1}{4}
%$$
%the equation is oscillatory; while if
%$$
%K^*=\limsup_{t \to +\infty} t \int_t^{+\infty} K(s) < \frac{1}{4}
%$$
%the equation is non-oscillatory. The borderline case $\frac{1}{4}$ cannot be detected. Our oscillation result, Theorem 6, is %effective in some situations of this borderline case (see Remark 6.48 of [BMR]).
%\end{Remark}

\section{Extensions of Calabi compactness criterion }
We shall now deal with \eqref{cauchypr} under the further assumption that $A$ is possibly negative.
Hereafter, we require the validity of \eqref{A1}, \eqref{sesto},
\eqref{V2}, \eqref{V3}, \eqref{F1} . Let $z\in
\lip_\loc(\R^+_0)$ be a solution of
\begin{equation}\label{ancoracauchy}
\left\{ \begin{array}{l} (vz')' + Avz = 0 \qquad
\text{on } \ \R^+,\\[0.2cm]
z(0^+) = z_0>0,
\end{array}\right.
\end{equation}
or of the analogous problem on $[r_0,+\infty)$. \par 

Choose a function $W\in L^\infty_\loc(\R^+_0)$ such that
\begin{equation}\label{ipotesiF}
W\ge 0 \quad \text{a.e. on } \R^+, \qquad W+A\ge 0 \quad
\text{a.e. on } \ \R^+.
\end{equation}
For instance, $W$ can be taken to be the negative part of $A$. To
apply the results of the previous section, we need to produce,
starting from \eqref{ancoracauchy} and $W$, a solution
$\widetilde{z}$ of a linear ODE of the type
$(\bar{v}\widetilde{z}')' +
\bar{A}\bar{v}\widetilde{z}=0$, for some new volume
function $\bar{v}$ and some $\bar{A}\ge 0$. Towards
this purpose, consider a solution $w(r) \in
\mathrm{Lip}_{\mathrm{loc}}(\R^+_0)$ of
\begin{equation}\label{cauchyconb2}
\left\{ \begin{array}{l} (vw')' -Wvw \ge 0 \qquad \mathrm{on \ }
\R^+ \\[0.2cm]
w(0^+) = w_0>0.
\end{array}\right.
\end{equation}

Note that from $$(vw')'\ge Wvw$$ we deduce $w'\ge 0$ a.e., hence $w$ has
a positive essential infimum on $\R_0^+$. Therefore, the function
$\widetilde{z} = z/w$ is well defined on $\R_0^+$ and solves
\begin{equation}\label{lineareperz}
\left\{ \begin{array}{l} \big([vw^2]\widetilde{z}'\big)' +
\big(A+W\big)[vw^2]\widetilde{z}\le 0 \qquad \text{on } \ \R^+\\[0.2cm]
\widetilde{z}(0) = z_0/w_0>0,
\end{array}\right.
\end{equation}
As observed in Remark \ref{indeboliamo}, the inequality sign in \eqref{lineareperz} is irrelevant for the proofs of $(3)$, $(4)$ of Theorem \ref{main}. In this way, $(3)$ and $(4)$ can be extended to cover sign-changing potentials by simply replacing $A$ with $A+W$, $v$ with $vw^2$ and $f$ with $f w^2$. The main problem therefore shifts to the search of explicit solutions $w$ of \eqref{cauchyconb2}, once $v$ and $W$ are given. \par
Up to taking some care when dealing with the initial condition, the same procedure can be carried on even when $v\equiv 1$. In this case, we are able to provide an explicit form for $w$ when the potential $W$ is a polynomial. This leads to the following theorem (see Theorem 6.41 of \cite{bmr2}). In the statement below, we denote with $I_\nu$ is the positive Bessel function of order $\nu$.

\begin{Theorem}[\textbf{Compactness with sign-changing curvature}]\label{compafinito}
Let $(M,\metric)$ be a complete m-dimensional Riemannian manifold. For each unit
speed geodesic $\gamma$ emanating from a fixed origin $o$, define
$$
K_\gamma(t) = \frac{\mathrm{Ricc}(\gamma',\gamma')(t)}{m-1}.
$$
Assume that one of the following set of assumptions is met.
\begin{itemize}
\item[$(i)$] The function $K_\gamma(t)$ satisfies
$$
K_\gamma(t) \ge -B^2\big(1+t^2\big)^{\alpha/2} \qquad \text{on } \
\R^+,
$$
for some $B>0$ and $\alpha\ge -2$ possibly depending on $\gamma$.
Having set
$$
0 \le  A_\gamma(t) = K_\gamma(t) + B^2\big(1+t^2\big)^{\alpha/2},
$$
suppose also that, for some $0<S<t$ such that $A_\gamma\not\equiv
0$ on $[0,S]$,
\begin{equation} \label{ancoragamma}
\begin{array}{l}
\disp
\int_S^t\left(\sqrt{A_\gamma(\sigma)}-\sqrt{\chi_{w^2}(\sigma)}\right)\di
\sigma \\[0.5cm]
\qquad \qquad \qquad
> \disp -\frac 12 \left(\log \int_0^S
A_\gamma(\sigma)w^2(\sigma)\di \sigma + \log\int_S^{+\infty}
\frac{\di \sigma}{w^2(\sigma)}\right),
\end{array}
\end{equation}
where
\begin{equation}\label{wespo}
w(t) = \left\{\begin{array}{ll}
\disp \sinh\left(\frac{2B}{2+\alpha}\left[(1+t)^{1+\frac{\alpha}{2}}-1\right]\right)
& \quad \text{if } \ \alpha \ge 0;\\[0.3cm]
\disp t^{1/2}I_{\frac{1}{2+\alpha}}\left(\frac{2B}{2+\alpha}t^{1+\frac{\alpha}{2}}\right)
& \quad \text{if } \ \alpha \in (-2,0); \\[0.3cm]
\disp t^{B'} & \quad \text{if } \ \alpha =-2,
\end{array}\right.
\end{equation}
and $B'= (1+\sqrt{1+4B^2})/2$.
\item[$(ii)$] The function $K_\gamma(t)$ satisfies
$$
K_\gamma(t) \ge \frac{B^2}{(1+t)^2} \qquad \text{on } \ \R^+,
$$
for some $B\in [0,1/2]$ possibly depending on $\gamma$. Having set
$$
0 \le A_\gamma(t) = K_\gamma(t) - \frac{B^2}{(1+t)^2},
$$
suppose also that, for some $0<S<t$ such that $A_\gamma\not\equiv
0$ on $[0,S]$, inequality \eqref{ancoragamma} holds with
\begin{equation}\label{wpoli}
w(t) = \left\{ \begin{array}{ll}
(1+t)^{B''}-(1+t)^{1-B''} & \quad \text{if } \ B \in [0,1/2); \\[0.3cm]
\sqrt{1+t}\log(1+t) & \quad \text{if } \ B=1/2,
\end{array}\right.
\end{equation}
and $B''= (1+\sqrt{1-4B^2})/2$.
\end{itemize}
Then, $M$ is compact and has finite fundamental group.
\end{Theorem}
\begin{Remark}
\emph{Note that, both for \eqref{wespo} and for
\eqref{wpoli}, the critical curve related to $w^2$ exists
since $1/w^2 \in L^1(+\infty)$.}
\end{Remark}
\begin{proof}
By Theorem \ref{teo_galloway}, it is enough to prove that, for every $\gamma$
issuing from $o$, the solution $g$ of
\begin{equation}\label{gODE}
\left\{ \begin{array}{l}
g''+K_\gamma(t)g = 0 \\[0.2cm]
g(0)=0, \quad g'(0)=1
\end{array}\right.
\end{equation}
has a first zero.\\
(i) A straightforward computation shows that the
function $w$ in \eqref{wespo} is a positive solution of
$$
w'' - B^2(1+t^2)^{\alpha/2}w\ge 0 \qquad \text{on } \R^+
$$
whose initial condition, in the cases $\alpha\in (-2,0)$ and
$\alpha\ge 0$, is
\begin{equation}\label{initialw}
w(0)=0, \qquad w'(0)=C>0.
\end{equation}
Consider $\widetilde{z}=g/w$. Then, $\widetilde{z}$ solves
\begin{equation}\label{perzbar}
(w^2 \widetilde{z}')' + A_\gamma w^2 \widetilde{z} \le 0 \qquad
\text{on } \R^+.
\end{equation}
In order to apply $(3)$ of Theorem \ref{main} to the differential
inequality \eqref{perzbar}, we shall make use of Remark
\ref{indeboliamo}. From \eqref{initialw}, in each of the cases of
\eqref{wespo} we obtain
\begin{equation}\label{passochiave!}
\frac{w^2 \widetilde{z}'}{\widetilde{z}}(0^+) =
\left(w^2\frac{g'}{g} - w w'\right)(0^+) = 0.
\end{equation}
We can thus apply $(3)$ of Theorem \ref{main}, and
\eqref{ancoragamma} implies that $\widetilde{z}$ (hence $g$)
has a first zero. Case $(ii)$ is analogous. Indeed, by Remark
\ref{rem_euler0}, $w$ in \eqref{wpoli} is a solution of
the Cauchy problem
$$
\left\{\begin{array}{l} \disp w'' +\frac{B^2}{(1+t)^2}w=0 \\[0.4cm]
g(0)=0, \quad g'(0)=C>0.
\end{array}\right.
$$
\end{proof}
\begin{Remark}
\emph{We recall that, by \eqref{notaint}, inequality
\eqref{ancoragamma} is equivalent to the somehow simpler
\begin{equation}\label{ineqprimozerosimply}
\int_S^t\sqrt{A_\gamma(\sigma)}\di \sigma > -\frac 12 \left(\log
\int_0^S A_\gamma(\sigma)w^2(\sigma)\di \sigma +
\log\int_t^{+\infty} \frac{\di \sigma}{w^2(\sigma)}\right).
\end{equation}
However, \eqref{ancoragamma} put in evidence that the RHS does not
depend on $t$, as opposed to conditions like
\eqref{calabialfinito} and \eqref{caselambdaneq1} where both $a$
and $b$ appear in the LHS as well as in the RHS. Furthermore, although somehow complicated, \eqref{ancoragamma} is entirely explicit once we are able to compute the critical curve related to $w^2$. In general, this can only be done numerically, but in some cases a closed expression can be given. For instance, this is so for $m=3$, $B=1/2$ in \eqref{wpoli}, for $B=0$ in \eqref{wpoli} and
for $\alpha=0,-2$ in \eqref{wespo}:
$$
\int_t^{+\infty} \frac{\di \sigma}{w^2(\sigma)} =
\left\{\begin{array}{ll} \disp
\frac{t^{-\sqrt{1+4B^2}}}{\sqrt{1+4B^2}} &
\quad \text{for } \eqref{wespo}, \ \alpha =-2 \text{ and for } B=0; \\[0.6cm]
\disp B^{-1} \big[\mathrm{coth}(Bt) -1 \big] & \quad \text{for } \eqref{wespo}, \ \alpha =0; \\[0.3cm]
\disp \frac{1}{\log(1+t)} & \quad \text{for }
\eqref{wpoli}, \ B=1/2, \ m=3.
\end{array}\right.
$$
Therefore, in the case $B=0$, \eqref{ineqprimozerosimply} reads
$$
\int_S^t \sqrt{K_\gamma(\sigma)}\di \sigma > - \frac 12 \left(\log
\int_0^S \sigma^2 K_\gamma(\sigma)\di \sigma - \log t \right),
$$
that should be compared to \eqref{calabialfinito}, while, for $\alpha=0$, \eqref{ineqprimozerosimply} becomes  
$$
\int_S^t \sqrt{K_\gamma(\sigma)+ B^2}\di \sigma > - \frac 12 \left(\log
\int_0^S K_\gamma(\sigma)\sinh^2(B\sigma)\di \sigma +
\log \frac{\coth(Bt)-1}{B}  \right),
$$
that should be compared to \eqref{caselambda1} and \eqref{caselambdaneq1}.}
\end{Remark}

Easier expressions can be obtained when considering oscillatory conditions. We state the result in analytic form.
\begin{Theorem}[\textbf{Generalized Calabi criterion}]\label{teo_calabi}
Let $K\in L^\infty_\loc(\R_0^+)$, and let $g\not\equiv 0$ be a
solution of $g''+Kg=0$. Then, $g$ oscillates in each of the
following cases:
\begin{itemize}
\item[$(1)$] $K$ satisfies
\begin{equation}\label{liminfricci}
K(t) \ge -B^2t^\alpha \qquad \text{when } \ t>t_0,
\end{equation}
for some $B>0$, $\alpha\ge -2$ and $t_0>0$, and the following
conditions hold:
\begin{equation}\label{condoscilla}
\begin{array}{l}
\text{for } \ \alpha=-2, \quad \disp \limsup_{t\ra +\infty}
\left(\int_{t_0}^t \sqrt{K(\sigma) +\frac{B^2}{\sigma^2}} \di
\sigma - \frac{\sqrt{1+4B^2}}{2}\log t\right) =
+\infty; \\[0.5cm]
\text{for } \ \alpha >-2, \quad \disp \limsup_{t\ra +\infty}
\left(\int_{t_0}^t \sqrt{K(\sigma)+B^2\sigma^\alpha} \di \sigma -
\frac{2B}{\alpha+2} t^{\frac{\alpha}{2}+1}\right) = +\infty.
%\alpha \ge 0, \qquad \disp \limsup_{r\ra +\infty} \left(\int_T^t
%\sqrt{K_\gamma(\sigma)+B^2(1+\sigma^2)^{\alpha/2}} \di \sigma -
%\frac{2B}{\alpha+2}
%(1+t)^{\frac{\alpha}{2}+1}-\frac{\alpha}{4}\log t\right) =
%+\infty.
\end{array}
\end{equation}
\item[$(2)$] $K$ satisfies
\begin{equation}\label{liminfricci22}
K(t) \ge \frac{B^2}{t^2} \qquad \text{when } \ t>t_0,
\end{equation}
for some $B\in [0, 1/2]$, $t_0>0$, and the following conditions
hold:
\begin{equation}\label{condoscilla22}
\begin{array}{l}
\text{for } \ B< \frac 12, \quad \disp \limsup_{t\ra +\infty}
\left(\int_{t_0}^t \sqrt{K(\sigma) -\frac{B^2}{\sigma^2}} \di
\sigma - \frac{\sqrt{1-4B^2}}{2}\log t\right) =
+\infty; \\[0.5cm]
\text{for } \ B = \frac 12, \quad \disp \limsup_{t\ra +\infty}
\left(\int_{t_0}^t \sqrt{K(\sigma) -\frac{1}{4\sigma^2}} \di
\sigma - \frac 12 \log\log t\right) = +\infty;
\end{array}
\end{equation}
\end{itemize}
\end{Theorem}
\begin{proof}
$(1)$. The equation $w''-B^2t^\alpha w=0$ on, say, $[1,+\infty)$ has the particular positive solution
\begin{equation}\label{leubuone}
\begin{array}{ll}
w(t) = \disp \sqrt{t}
I_{\frac{1}{2+\alpha}}\left(\frac{2B}{2+\alpha}t^{1+\frac{\alpha}{2}}\right)
& \quad \text{if } \ \alpha>-2; \\[0.5cm]
w(t) = \disp t^{B'}, \quad B' = \frac{1+\sqrt{1+4B^2}}{2} & \quad
\text{if } \ \alpha=-2,
\end{array}
\end{equation}
where $I_\nu(t)$ is the Bessel function of order $\nu$. From 
$$
\qquad I_\nu(t)= \frac{e^t}{\sqrt{2\pi t}}(1+o(1))\qquad \text{as
} \ t\ra +\infty
$$
(see \cite{lebedev}, p. 102), in
both cases $\alpha=-2$ and $\alpha>-2$ we deduce that $1/w^2 \in L^1(+\infty)$. Moreover,
\begin{equation}\label{asintobuone}
\int_t^{+\infty} \frac{\di \sigma}{w^2(\sigma)} \sim \left\{
\begin{array}{ll}
%Ct^{-\alpha/2}
%\exp\left(-\frac{4B}{2+\alpha}(1+t)^{1+\frac{\alpha}{2}}\right) &
%\quad \text{if } \ \alpha\ge 0; \\[0.5cm]
C \exp\left(-\frac{4B}{2+\alpha}t^{1+\frac{\alpha}{2}}\right)
& \quad \text{if } \ \alpha >-2; \\[0.5cm]
C t^{1-2B'} = Ct^{-\sqrt{1+4B^2}} & \quad \text{if } \ \alpha=-2.
\end{array}\right.
\end{equation}
Since the function $\widetilde{z}=g/w$ solves 
$$
(w^2 \widetilde{z}')' + (K+ B^2 t^\alpha) w^2 \widetilde{z} \le 0 \qquad \text{on } [1,+\infty), 
$$
by $(4)$ of Theorem \ref{main}, $z$ (and hence $g$) oscillates provided 
$$
\limsup_{t\ra +\infty} \int_{t_0}^t \Big( \sqrt{K(\sigma) + B^2 \sigma^\alpha}- \sqrt{\chi_{w^2}(\sigma)} \Big) \di \sigma = +\infty
$$
 which, by  Remark \ref{remimpo}, is equivalent to 
\begin{equation}\label{oscillagenerale}
\limsup_{t\ra +\infty} \int_{t_0}^t \sqrt{K(\sigma) + B^2 \sigma^\alpha}\di \sigma  + \frac 12 \log \int_t^{+\infty}\frac{\di \sigma}{w^2(\sigma)} = +\infty
\end{equation}
By \eqref{asintobuone}, conditions \eqref{condoscilla} and
\eqref{oscillagenerale} are equivalent, thus the conclusion.\\
$(2)$. The proof is the same. Indeed, it is enough to consider the
following positive solution $w$ of $w''+B^2t^{-2}w=0$:
\begin{equation}\label{lesolitepositive}
\begin{array}{ll}
w(t) = \disp t^{B''}, \quad B'' = \frac{1+\sqrt{1-4B^2}}{2} &
\quad
\text{if } \ B \in [0,1/2); \\[0.5cm]
w(t) = \disp \sqrt{t} \log t & \quad \text{if } \ B=1/2.
\end{array}
\end{equation}
Again, in both cases $1/w^2\in L^1(+\infty)$.
\end{proof}

\begin{Remark}
\emph{Note that, for $B=0$, we recover another proof of the original Calabi oscillation criterion, which is different from that described in the previous section.} 
\end{Remark}
Polynomial lower bounds for $K$ are clearly chosen for their simplicity. Indeed, the statement in its full generality only requires a positive solution $w$ of $w'' + Ww\ge 0$, where the weight $W$ has only to satisfy $K+W \ge 0$. In this way, arbitrary lower bounds for $K$ are allowed, and up to finding a suitable positive $w$ the oscillatory conditions are explicit. This improves on Moore oscillation criterion, where the existence of the limit in \eqref{condimoore} is essential for the proof of Theorem \ref{moorenehari} to work. The same discussion holds for Theorem \ref{compafinito}, up to the further requirement that $w$ is sufficiently well-behaved as $t\ra 0^+$. From this perspective, Theorem \ref{compafinito} improves on Theorem \ref{pmg}, whose proof seems to us to be hardly generalizable when the lower bound for $K$ is nonconstant.\\ 
\par
The procedure described above, which loosely speaking allows to translate the potential up to inserting a weight, can be iterated. In this way, we can obtain finer and finer criteria in a very simple way. We now describe how to proceed in this direction. The first example is the following
\begin{Theorem}[\textbf{Positivity and nonoscillation
criteria}]\label{nuovinonoscillatori}
Let $K \in L^\infty_\loc(\R^+_0)$.
\begin{itemize}
\item[$(1)$] Suppose that
\begin{equation}\label{newboundK1}
K(t) \le \frac{1}{4(1+t)^2}\left[1+\frac{1}{\log^2(1+t)}\right]
\qquad \text{on } \ \R^+.
\end{equation}
Then, every solution $g$ of
\begin{equation}\label{nonono}
\left\{\begin{array}{l} g''+K(t)g\ge 0 \\[0.2cm]
g(0)=0, \quad g'(0)=1
\end{array}\right.
\end{equation}
is positive on $\R^+$ and satisfies $g(t) \ge C\sqrt{t\log t}
\log\log t$, for some $C>0$ and for $t>3$.
\item[$(2)$] Suppose that
\begin{equation}\label{newboundK}
K(t) \le \frac{1}{4t^2}\left[1+\frac{1}{\log^2t}\right] \qquad
\text{on } \ [t_0,+\infty),
\end{equation}
for some $t_0>0$. Then, every solution $g$ of $g''+Kg=0$ is
nonoscillatory.
\end{itemize}
\end{Theorem}
\begin{proof}
$(1)$. By Sturm argument, it is sufficient to prove the desired
conclusion under the additional assumptions that $g$ satisfies
\eqref{nonono} with the equality sign, and that
$$
K(t) \ge \frac{1}{4(1+t)^2}.
$$
Let $w(t)=\sqrt{1+t}\log(1+t)$ be the solution of \eqref{nonono} with the equality sign and with
$K(t)= [4(1+t)^2]^{-1}$. Then, $\widetilde{z}=g/w$ solves
\begin{equation}
\left\{ \begin{array}{l} \disp (w^2\widetilde{z}')' +
\left[ K(s)-\frac{1}{4(1+t)^2}\right]w^2\widetilde{z} = 0 \quad \text{on } \R^+\\[0.4cm]
\widetilde{z}(0)=1, \qquad \widetilde{z}'(0)=0.
\end{array}\right.
\end{equation}
Applying $(1)$ of Theorem \ref{main}, $\widetilde{z}$ is
positive provided
$$
K(t) - \frac{1}{4(1+t)^2} \le \chi_{w^2}(t) =
\frac{1}{4(1+t)^2\log^2(1+t)},
$$
which is \eqref{newboundK1}, and $\widetilde{z}$ satisfies
$$
\widetilde{z}(t) \ge -C \sqrt{\int_t^{+\infty} \frac{\di
\sigma}{w^2(\sigma)}} \log \int_t^{+\infty} \frac{\di
\sigma}{w^2(\sigma)} = C\frac{\log\log t}{\sqrt{\log t}},
$$
for some $C>0$. The lower bound for $g$ follows at once by the
definition of
$\widetilde{z}$.\\
To prove $(2)$, again by Sturm argument we can assume that the
inequality $K(t)\ge 1/[4t^2]$ holds. Proceeding along
the same lines as for $(1)$ with the choice $w= \sqrt{t}\log t$, and using $(2)$ of Theorem \ref{main}, we reach the desired
conclusion.
\end{proof}

The next prototype case illustrates the sharpness of our criteria. Let
$$
K(t) = \frac{1}{4t^2} + \frac{c^2}{4t^2 \log^2t}, \qquad \text{on } [2,+\infty), 
$$
where $c>0$ is a constant. Then, applying $(2)$ of Theorem \ref{teo_calabi}, case $B=1/2$ we deduce that $g$ oscillates whenever $c>1$. On the other hand, if $c\le 1$, by Theorem \ref{nuovinonoscillatori} $g$ is nonoscillatory. However, on $[2,+\infty)$ 
$$
\frac 14 < k(t) = t \int_t^{+\infty}K(\sigma)\di \sigma \le \frac 14 +
t\frac{c^2}{4t}\int_t^{+\infty}
\frac{\di\sigma}{\sigma\log^2\sigma} = \frac 14 + \frac{c^2}{4\log
t},
$$
hence the Hille-Nehari criterion cannot detect neither the
oscillatory nor the nonoscillatory behaviour of $g$ depending on
$c$. Similarly, also Moore criterion is not sharp enough. The proof of Theorem \ref{nuovinonoscillatori} suggests an iterative improving procedure. In the general case, suppose that we are given an ordinary differential equation of the type $(vz')'+Avz=0$, with $v$ such that $\chi$ can be defined. By Sturm argument, there is no loss of generality if we assume that $A \ge \chi$. An explicit solution $w$ of 
$$
(vw')' + \chi vw =0
$$
is given by 
$$
w(t) = -\sqrt{\int_t^{+\infty} \frac{\di s}{v(s)}} \log \int_t^{+\infty} \frac{\di s}{v(s)},
$$
and it is positive on some intervall $[r_0,+\infty)$. Then, $\widetilde{z}= z/w$ solves 
$$
(\bar{v} \widetilde{z}')' + (A-\chi)\bar{v}\widetilde{z}=0 \qquad \text{on } [r_0,+\infty), 
$$
where $\bar{v}=vw^2$, which implies that $\widetilde{z}$, and therefore $z$, are nonoscillatory if $(vw^2)^{-1} \in L^1(+\infty)$ and 
$$
A(r) - \chi(r) \le \chi_{vw^2}(r), 
$$
and oscillatory if  $(vw^2)^{-1} \in L^1(+\infty)$ and 
$$
\limsup_{t \ra +\infty} \int_{t_0}^t \Big(\sqrt{A(s)-\chi(s)} - \sqrt{\chi_{vw^2}(s)}\Big)\di s = +\infty.
$$
Now, the procedure can be pushed a step further by considering $\widetilde{z}$. This enables us to construct finer and finer critical curves. As an example, we refine Theorem \ref{nuovinonoscillatori}. Suppose that 
$$
K(t) \ge \frac{1}{4t^2} + \frac{1}{4t^2 \log^2 t}
$$
on, say, $[2,+\infty)$. Then, as in the proof of Theorem \ref{nuovinonoscillatori}, define $w= \sqrt{t}\log t$ and $v=w^2 = t \log^2 t$. Since $w$ is a positive solution of $w'' + (4t^2)^{-1}w=0$ on some $[r_1,+\infty)$, $z=g/w$ is well defined and solves $(vz')'+Avz=0$ on $[r_1,+\infty)$, where 
$$
A(t) = K(t) - \frac{1}{4t^2} \ge \frac{1}{4t^2\log^2t} = \chi_{w^2}(t) = \chi(t).
$$
Now, the function
$$
w_2(t) = -\sqrt{\int_t^{+\infty} \frac{\di s}{v(s)}} \log \int_t^{+\infty} \frac{\di s}{v(s)} = \frac{\log \log t}{\sqrt{\log t}}
$$
is a solution of $(vw_2')'+ \chi vw_2=0$, positive after some $r_2\ge r_1$. Setting 
$$
v_2(t)= v(t)w_2(t)^2 = t \log t \log^2 \log t, 
$$
then
$$
\frac{1}{v_2(t)} \in L^1(+\infty),
$$
and the function $z_2 = z/w_2$ is a solution of $(v_2z_2')' + A_2 v_2 z_2 =0$ on $[r_2,+\infty)$, where
$$
A_2(t) = A(t) - \chi(t) = K(t) - \frac{1}{4t^2} - \frac{1}{4t^2 \log^2 t}  \ge 0.
$$
Thus $z_2$, and hence $z$ and $g$, is nonoscillatory provided
$$
A_2(t) \le \chi_{v_2}(t), \qquad \text{that is,} \qquad K(t) \le \frac{1}{4t^2} + \frac{1}{4t^2 \log^2 t} + \frac{1}{4t^2 \log^2 t \log^2 \log t},
$$
and, by \eqref{nuova}, it is oscillatory if
$$
\limsup_{t\ra +\infty} \left(\int_{t_2}^{t} \sqrt{K(\sigma) - \frac{1}{4\sigma^2} - \frac{1}{4\sigma^2 \log^2 \sigma}}\di \sigma 
- \frac 12 \log\log\log t\right) = +\infty.
$$

\begin{Remark}
\emph{We mention that, with the aid of the change of variables \eqref{tr} and \eqref{zr}, Theorems \ref{compafinito}, \ref{teo_calabi} and \ref{nuovinonoscillatori} can be applied to get sharp extensions of index estimates for stationary Schr\"odinger operators on $\R^m$, $m\ge 3$, that highly improve on classical results of M. Reed and B. Simon \cite{reedsimon3}, and W. Kirsch
and B. Simon \cite{kirschsimon}. The interested reader can consult \cite{bmr2}, Theorem 6.50.}
\end{Remark}

\bibliographystyle{amsplain}
\bibliography{biblio_oscillation}

\end{document}